\documentclass[12pt]{amsart}

\usepackage{amssymb,latexsym}
\usepackage[enableskew]{youngtab}
\usepackage{verbatim,euscript}
\usepackage{fullpage}
\usepackage[all]{xy}

\usepackage[colorlinks=true,linkcolor=blue,urlcolor=violet,citecolor=magenta]{hyperref}

\input {cyracc.def}
\numberwithin{equation}{section}

\newtheorem{thm}{Theorem}[section]

\theoremstyle{remark}
\newtheorem{rmk}[thm]{Remark}
\newtheorem*{exa}{Example}

\theoremstyle{definition}
\newtheorem{ex}[thm]{Example}

\newenvironment{proof*}
{\noindent {\sl Proof.}\quad }{\hfill $\square$}

\newcommand {\ah}{{\mathfrak a}}
\newcommand {\be}{{\mathfrak b}}

\newcommand {\g}{{\mathfrak g}}

\newcommand {\n}{{\mathfrak n}}

\newcommand {\p}{{\mathfrak p}}

\newcommand {\ut}{{\mathfrak u}}

\newcommand {\slno}{{\mathfrak{sl}}_{n+1}}

\newcommand {\spn}{{\mathfrak{sp}}_{2n}}

\newcommand {\esi}{\varepsilon}
\newcommand {\ap}{\alpha}

\newcommand {\lb}{\lambda}

\newcommand {\HW}{\widehat W}
\newcommand {\HV}{\widehat V}
\newcommand {\HP}{\widehat\Pi}
\newcommand {\HD}{\widehat\Delta}

\newcommand {\ck}{{\mathcal K}}

\newcommand {\cku}{\hat{\mathcal K}}

\newcommand {\rk}{{\mathsf{rk\,}}}

\newcommand {\GR}[2]{{\textrm{{\bf #1}}}_{#2}}

\newcommand {\un}{\underline}

\newcommand {\Ab}{\mathfrak{Ab}}

\newcommand {\beq}{\begin{equation}}
\newcommand {\eeq}{\end{equation}}

\newcommand{\curle}{\preccurlyeq}
\renewcommand{\le}{\leqslant}
\renewcommand{\ge}{\geqslant}

\newcommand{\adn}{{\sf ad}-nilpotent }

\newcommand{\eus}{\EuScript}

\newfam\Bbbfam\newfam\eufam\newfam\eusfam%
\font\Bbbfont=msbm10 scaled 1200%
\font\Bbbsmallfont=msbm8%
\textfont\Bbbfam=\Bbbfont\scriptfont\Bbbfam=\Bbbsmallfont

\begin{document}
\setlength{\parskip}{2pt plus 4pt minus 0pt}
\hfill {\scriptsize August 11, 2010} 
\vskip1.5ex

\title[Abelian ideals and Dynkin diagrams]{Abelian ideals of a Borel subalgebra and
subsets of the Dynkin diagram}
\author{Dmitri I. Panyushev}
\address[]{Independent University of Moscow,
Bol'shoi Vlasevskii per. 11, 119002 Moscow, \ Russia
\hfil\break\indent
Institute for Information Transmission Problems, B. Karetnyi per. 19, Moscow 
127994}
\email{panyush@mccme.ru}
\begin{abstract}

Let $\g$ be a simple Lie algebra  
and  $\Ab(\g)$  the set of Abelian ideals of a Borel subalgebra of $\g$.
In this note,  an interesting connection between $\Ab(\g)$ and the subsets of the 
Dynkin diagram of $\g$ is discussed. We notice that
the number of abelian ideals with $k$ generators equals the number of 
subsets of the Dynkin diagram with $k$ connected components.
For $\g$ of type $\GR{A}{n}$ or  $\GR{C}{n}$, we provide a combinatorial explanation of this coincidence by constructing a suitable bijection.
We also construct a general bijection between $\Ab(\g)$ and the 
subsets of the Dynkin diagram, which is based on the  theory
developed by Peterson and Kostant.
\end{abstract}
\maketitle

\section*{Introduction}

\noindent
Let $\g$ be a complex simple Lie algebra with a Borel subalgebra $\be$.
The set of abelian ideals of $\be$, denoted $\Ab(\g)$, attracted much attention after appearance of \cite{ko98},
where Kostant popularised (and elaborated on) a remarkable result of D.\,Peterson
to the effect that $\#\Ab(\g)=2^{\rk\g}$.
The aim of this note is to report on a surprising connection between $\Ab(\g)$ and the subsets of the Dynkin diagram of $\g$.
Namely, comparing independently performed computations \cite{djok-et,coveri}, we notice that
the number of abelian ideals with $k$ generators equals the number of 
subsets of the Dynkin diagram with $k$ connected components 
(see details in Section~\ref{sect:odin}).
For $\g$ of type $\GR{A}{n}$ or  $\GR{C}{n}$, we provide a combinatorial explanation of this coincidence by constructing a suitable bijection between $\Ab(\g)$ and the subsets of
the Dynkin diagram  (see Sect.~\ref{sect:A_n}).
In Section~\ref{sect:bij}, we construct a general bijection between $\Ab(\g)$ and the 
subsets of the Dynkin diagram. Although this last bijection does not respect
the number of generators and connected components, we believe it is interesting in
its own right. This exploits a relationship between the abelian ideals and certain elements of 
the affine Weyl group of $\g$ \cite{ko98}. 

We refer to \cite{hump} for standard results on root systems and affine Weyl groups.

{\small
{\bf Acknowledgements.} This work was done during my stay at 
the Max-Planck-Institut f\"ur Mathematik (Bonn). I would like to thank the Institute for
the hospitality and inspiring environment. 
}

\section{An empirical observation}
\label{sect:odin}

\noindent
Let $\Delta$ be the root system of $\g$ and  $\Delta^+$ the subset of
positive roots corresponding to $\be$. Then $\Pi=\{\ap_1,\dots,\ap_n\}$ is the set of simple 
roots in $\Delta^+$. We regard $\Delta^+$ as poset with respect to the  {\it root order}.
This means that $\nu\curle\mu$ if $\mu-\nu$ is a non-negative integral linear combination
of simple roots.

An  ideal $\ah$ of $\be$ is said to be {\it abelian}, if $[\ah,\ah]=0$.
Then $\ah$ is a sum of certain root spaces in $\ut=[\be,\be]$,  i.e.,
$\ah=\bigoplus_{\gamma\in I}\g_\gamma$.  
Here $I$ is necessarily an {\it upper ideal\/} of $\Delta^+$, i.e., if $\nu\in I, \mu\in\Delta^+$,
and $\nu+\mu\in \Delta^+$, then $\nu+\mu\in I$. In other words, if $\nu\in I$ and $\nu\curle
\gamma$, then $\gamma\in I$.
The property of being abelian means that
$\gamma'+\gamma''\not\in \Delta^+$ for all $\gamma',\gamma''\in I$.
Let $\Ab=\Ab(\g)$ be the poset, with respect to inclusion, of  all abelian ideals.
We will mostly work in the combinatorial setting, so that an abelian ideal $\ah$
is identified with
the corresponding set $I$ of positive roots. The minimal elements (roots) of $I$
are also called the {\it generators} of $I$. 

Let $\kappa(I)$ be the number of minimal elements of $I$.  
The generating function
\[
   \cku_\Ab(q):=\sum_{I\in\Ab} q^{\kappa(I)}
\]
is called the {\it upper covering polynomial} (of the poset $\Ab$). We refer to
\cite{coveri} for generalities on covering polynomials. In fact, there is also a {\it lower
covering polynomial}, which is not considered here.
The polynomials $\cku_{\Ab(\g)}(q)$ are known for all simple Lie algebras $\g$,
see \cite[Section\,5]{rodstv}  and \cite[Section\,5]{coveri}.
By the very definition, the coefficient of $q^k$ is the number of abelian ideals with 
$k$ generators. 

Recently, we have discovered that  the polynomials $\cku_{\Ab(\g)}(q)$ had  
another interpretation in terms of the Dynkin diagram of $\Delta$.
Regarding $\Pi$ as the set of nodes in the Dynkin diagram, we say that a
subset of  $\Pi$ is {\it connected\/} if it is connected in the Dynkin diagram.
Then, for any subset of $\Pi$, we can consider the number of its connected components.
Let $N_k=N_k(\Delta)$ denote the number of subsets of $\Pi$ with exactly $k$ connected
components. Then, of course, $\sum_{k\ge 0} N_k=2^n$.
For all $\Delta$, the numbers $N_k$ are found in \cite[n$^o$\,2]{djok-et}.\footnote{ 
More precisely, there is only a recursive formula for $N_k(\GR{D}{n})$ in \cite{djok-et}.
But it is equivalent to the recursive formula for polynomials $\cku_{\Ab}(q)$, cf. 
\cite[Eq. (5.1)]{coveri} and \eqref{recurs}  below. }
Comparing them with our upper covering polynomials, we get
the striking assertion:   

\begin{thm}   \label{nabl-empirical}
For any reduced irreducible root system $\Delta$ (i.e., for any simple Lie algebra $\g$), 
we have
\[
     \sum_{k\ge 0} N_k q^k =\cku_{\Ab(\g)}(q) .
\]
In other words, the number of abelian ideals with $k$ generators equals the number 
of subsets of\/ $\Pi$ with $k$ connected components.
\end{thm}

\noindent
Actually, one of the goals of \cite{djok-et} is to classify the closed subsets $P$ of
$\Delta$ such that $\Delta\setminus P$ is also closed. Such a $P$ is said to be 
{\it invertible}.  If $P$ is invertible, then so is $w(P)$ for any $w\in W$. 
Let $N(\Delta)$ be the number of $W$-orbits in the set of all invertible subsets of $\Delta$.
It is shown in \cite[Eq.\,(2)]{djok-et} that
\[
   N(\Delta)=\sum_{k\ge 0} N_k 2^k .
\] 
In this way, one obtains a surprising interpretation of the value $\cku_{\Ab(\g)}(2)$.
For the reader's convenience, we reproduce a table with all these polynomials.

\begin{table}[htb]     
\caption{ The upper covering 
polynomials for $\Ab(\g)$}  \label{table:ab}
\begin{tabular}{c|l}
\phantom{q}$\g$ & \phantom{quq}$\cku_{\Ab(\g)}(q)$  
\\ \hline
$\GR{A}{n}$, $\GR{B}{n}$, $\GR{C}{n}$& $\underset{k\ge 0}{\sum}\genfrac{(}{)}{0pt}{}{n+1}{2k}q^k$ 
{\rule{0pt}{2.8ex}} 
\\
$\GR{D}{n}$ & 
$\underset{k\ge 0}{\sum}\bigl(\genfrac{(}{)}{0pt}{}{n+2}{2k}-
4\genfrac{(}{)}{0pt}{}{n-1}{2k-2}\bigr )q^k$ {\rule{0pt}{3ex}}  
\\
$\GR{E}{6}$ & $1{+}25q{+}\phantom{1} 27q^2{+}11q^3$  {\rule{0pt}{2.4ex}}  
\\
$\GR{E}{7}$ & $1{+}34q{+}\phantom{1} 60q^2{+}30q^3{+}\phantom{1} 3q^4$ 
{\rule{0pt}{2.3ex}}  
\\
$\GR{E}{8}$ & $1{+}44q{+}118q^2{+}76q^3{+}17q^4$ {\rule{0pt}{2.3ex}}   
\\
$\GR{F}{4}$ & $1{+}10q{+}\phantom{11} 5q^2$  
\\
$\GR{G}{2}$ & 
{$1{+}\phantom{1} 3q$}    
\\ \hline
\end{tabular} 
\end{table}%

\noindent
In \cite[Section~5]{coveri}, we observed that if the Dynkin diagram has no branching 
nodes, then
$\cku_{\Ab(\g)}$ depends only on $\rk(\g)$, i.e., on the number of nodes. For instance,
the upper covering polynomial for $\GR{F}{4}$ 
(resp. $\GR{G}{2}$) is equal
to that for $\GR{A}{4}$ (resp. $\GR{A}{2}$).
Having at hand Theorem~\ref{nabl-empirical}, we now realise that the reason is that 
the  connected components of a subset of $\Pi$
does not depend on the length of simple roots. 
 
\noindent
There are some regularities in Table~\ref{table:ab}.
For all classical series, these polynomials  satisfy the recurrence relation 
\begin{equation}          \label{recurs}
     \cku_{\Ab(\GR{X}{n})}(q)=2\cku_{\Ab(\GR{X}{n-1})}(q)+(q-1)\cku_{\Ab(\GR{X}{n-2})}(q)\ , 
\end{equation}
where $\GR{X}{}\in \{\mbox{\bf A,B,C,D}\}$. 
Furthermore, the sequence $\GR{E}{3}=\GR{A}{2}\times \GR{A}{1}$, 
$\GR{E}{4}=\GR{A}{4},\GR{E}{5}=\GR{D}{5}$, $\GR{E}{6}$, $\GR{E}{7},\GR{E}{8}$
can be regarded as the `exceptional' series, and for this series the same
recurrence relation holds. Comparing the coefficients of $q^k$ in \eqref{recurs},
one obtains the relation
\beq   \label{eq:recurs2}
    N_k(\GR{X}{n})=2N_k(\GR{X}{n-1})+N_{k-1}(\GR{X}{n-2})-N_k(\GR{X}{n-2}) .
\eeq
That is, it is true not only for $\GR{D}{n}$, as pointed out in \cite[p.\,341]{djok-et}, but for all
our series, including the exceptional one. Actually, relation \eqref{eq:recurs2} for the 
number of subsets with prescribed number of connected components
remains true if we extend any finite
graph $\eus G_{n-2}$ with a chain of length 2, see the pattern below:

\begin{figure}[htbp]
\begin{picture}(200,30)(-20,-5)
\multiput(30,12)(20,0){2}{\circle{6}}
\multiput(33,12)(20,0){2}{\line(1,0){14}}
\put(70,10){\framebox{$\eus G_{n-2}$}}
\put(-10,9){$\eus G_n$:}
\put(45,0){$\underbrace%
{\mbox{\hspace{60\unitlength}}}_{\eus G_{n-1}}$}
\end{picture}
\end{figure}

\noindent
We leave it to the reader to prove \eqref{eq:recurs2} for $\GR{X}{n}=\eus G_n$.

\begin{rmk}
For a sequence of polynomials $\ck_n(q)$ satisfying relation \eqref{recurs}, we have
$\ck_n(-1)=2\ck_{n-1}(-1)-2\ck_{n-2}(-1)$. This yields a kind of
$4$-periodicity for the values at $q=-1$: \ 
$\ck_{n+2}(-1)=-4 \ck_{n-2}(-1)$.
\end{rmk}

\section{A good bijection for $\GR{A}{n}$ and $\GR{C}{n}$}
\label{sect:A_n}

\noindent
In what follows, we write $2^\Pi$ for the set of all subsets
of $\Pi$. Theorem~\ref{nabl-empirical}  suggests that there could 
be a natural one-to-one correspondence between $\Ab(\g)$ and $2^\Pi$, under which 
the ideals with $k$ generators correspond to the subsets with
$k$ connected components.  We call it a {\it good bijection}.
So far, we did not succeed in finding such a good bijection in general. 
In fact, we are able to construct a general  bijection 
$\Ab(\g)\stackrel{1:1}\longrightarrow 2^\Pi$ (see Section~\ref{sect:bij}), 
but that bijection is not good. 

In this section,  a good bijection is constructed
for $\g=\slno$ or $\spn$.

Let $\Pi=\{\ap_i=\esi_i-\esi_{i+1} \mid i=1,2,\dots,n\}$,
be the standard set of simple roots for $\GR{A}{n}$. 
We regard $\Pi$ as the $n$-element interval: $[n]:=\{1,2,\dots,n\}$.
Every positive root $\gamma$ of $\GR{A}{n}$  is of the form 
$\gamma=\ap_i+\ap_{i+1}+ \dots +\ap_j$ with
$i\le j$,
and therefore we identify it with the subset ({\it interval}) $[i,j]:=\{i,i+1,\dots,j\}$ of
$[n]$. In the usual terminology on root systems, 
$[i,j]$ is the {\it support\/} of $\gamma$, also denoted $\textrm{supp}(\gamma)$.

Now, let $I$ be an abelian ideal in $\Delta^+(\GR{A}{n})$ and 
$\gamma_1,\dots, \gamma_k$ the set of generators of $I$. 
Then such an ideal is also denoted by $I(\gamma_1,\dots,\gamma_k)$.
(Of course, this imposes certain restrictions on $\gamma_i$'s, which we describe
below.)

Let  $\Phi: \Ab(\slno) \to \{\text{subsets of $[n]$}\}=:2^{[n]}$ be defined 
by the formula:
\[
     I=I(\gamma_1,\dots,\gamma_k) \mapsto \textrm{supp}(\gamma_1)\oplus \dots
    \oplus \textrm{supp}(\gamma_k) ,
\]
where `$\oplus$' stands for the ``exclusive disjunction'' (or ``addition mod 2'') in the Boolean 
algebra of subsets of  $[n]$.
 
\begin{thm}   \label{thm:An}
The map $\Phi$ sets up a one-to-one correspondence between $\Ab(\slno)$ and 
$2^{[n]}$. Moreover, if $I$ has $k$ generators, then $\Phi(I)$ has $k$ connected  components.
\end{thm}
\begin{proof}
Suppose we are given $k$ positive roots $\gamma_s=[i_s,j_s]$, $s=1,\dots,k$.
Without loss of generality, we can assume that $i_1\le i_2\le\dots\le i_k$.
It is then easily seen that $\{\gamma_1,\dots,\gamma_k\}$ is the set of generators of an
abelian ideal if and only if 
\beq    \label{eq:generators}
    1\le i_1 <i_2< \dots <i_k\le j_1<j_2< \dots < j_k\le n . 
\eeq
(This presentation also shows that $\#\Ab(\slno)=2^n$.)
Thus, we have $2k$ points in the whole interval $[n]$ and  $2k-1$ intervals between
them.
Now, a  straightforward verification shows that 
    $[i_1,j_1]\oplus \dots \oplus [i_k,j_k]$ is the  union of the following disjoint intervals:
\begin{itemize}    
\item we begin with the interlacing $i$-intervals: $[i_1,i_2-1],\,[i_3,i_4-1],\dots$;
\item we end up with the interlacing $j$-intervals: 
          $\dots,  [j_{k-3}+1,j_{k-2}],\,[j_{k-1}+1,j_k]$; 
\item  if $k$ is odd, then we also take
         the middle interval $[i_k,j_1]$.
\end{itemize}
The total number of such intervals equals $k$, as required.

Conversely, any collection of $k$ disjoint intervals allows us to write up a sequence of the
form \eqref{eq:generators} and obtain an abelian ideal.
\end{proof}

\begin{exa}
For $k=3$, we obtain the intervals  $[i_1,i_2-1],\,[i_3,j_1],\, [j_2+1,j_3]$.\\
For $k=4$, we obtain the intervals  $[i_1,i_2-1],\,[i_3,i_4-1],\, [j_1+1,j_2],\,[j_3+1,j_4]$.
\end{exa}

To construct a good bijection for $\g=\spn$, 
we use the usual unfolding $\GR{C}{n}\leadsto \GR{A}{2n-1}$ (see figure below) and 
combine it with the above $\mathfrak{sl}$-algorithm.

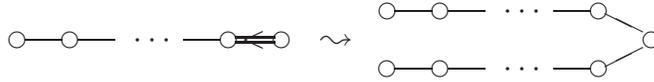
\begin{figure}[htb]
\begin{picture}(300,30)(20,0)
\multiput(30,12)(20,0){2}{\circle{6}}
\multiput(110,12)(20,0){2}{\circle{6}}
\multiput(112.5,11)(0,2){2}{\line(1,0){15}}
\multiput(33,12)(20,0){2}{\line(1,0){14}}
\put(93,12){\line(1,0){14}}
\put(74,9){$\cdots$}
\put(115,9){$<$} 
\put(145,9){$\leadsto$}

\multiput(170,1)(20,0){2}{\circle{6}}
\multiput(173,1)(20,0){2}{\line(1,0){14}}
\put(214,-2){$\cdots$}
\multiput(170,23)(20,0){2}{\circle{6}}
\multiput(173,23)(20,0){2}{\line(1,0){14}}
\put(214,20){$\cdots$}
\multiput(250,1)(0,22){2}{\circle{6}}
\multiput(233,1)(0,22){2}{\line(1,0){14}}
\put(270,12){\circle{6}}
\put(252.5,2){\line(2,1){14.2}}
\put(252.5,22){\line(2,-1){14.2}}

\end{picture}
\end{figure}

\noindent
If $\Delta$ is of type $\GR{C}{n}$, then $\Pi=\{\esi_1-\esi_2,\dots,\esi_{n-1}-\esi_n, 2\esi_n\}$ and the unique maximal abelian ideal consists of the roots $\{\esi_i+\esi_j\mid 1\le i\le j\le n\}$.
The positive roots of $\GR{A}{2n-1}$ are identified with the intervals of $[2n-1]$, as above.
Under the above unfolding,  a short root $\esi_i+\esi_j$ ($i\ne j$) is repalced with two
roots $[i,2n-j]$ and $[j,2n-i]$ of $\GR{A}{2n-1}$; and a long root $2\esi_i$ is replaced with
one root $[i,2n-i]$. 

If $I=I(\gamma_1,\dots,\gamma_k)$ is an abelian ideal of $\Delta^+(\GR{C}{n})$, 
then we do the following:
\begin{itemize}
\item[\sf (i)] \  Replace each $\gamma_i$ with one or two roots (intervals) for $\GR{A}{2n-1}$,
as explained.
\item[\sf (ii)] \  Take the sum modulo 2 of all these intervals. Obviously, the resulting subset 
of  $[2n-1]$, $\tilde \Phi(I)$, is symmetric with respect to the middle point $\{n\}$. 
\item[\sf (iii)] \  Take the quotient of $\tilde \Phi(I)$ by this symmetry, i.e., consider
$\Phi(I):=\tilde \Phi(I)\cap [n]$.
\end{itemize}

\noindent
In this way, we obtain a mapping $\Phi:\Ab(\spn)\to 2^{[n]}$, 
and it is not hard to verify that it is a good
bijection.

\section{A general bijection}
\label{sect:bij}

\noindent
In this section, a general bijection between $\Ab(\g)$ and  $2^{\Pi}$ is constructed.
To this end, we need a parametrisation of the abelian ideals described by Kostant \cite{ko98}, which  relies on the relationship, due to D.\,Peterson, between the abelian ideals and the so-called {\it minuscule elements\/} of the affine Weyl group of $\Delta$.
Recall the necessary setup.

We have the real vector space $V=\oplus_{i=1}^n{\mathbb R}\ap_i$, 
the usual Weyl group generated by the reflections
$s_1,\dots,s_n$,  and a $W$-invariant inner product $(\ ,\ )$ on $V$. Then

$Q=\oplus _{i=1}^n {\mathbb Z}\ap_i \subset V$ is the root lattice; 

$Q^+=\{\sum_{i=1}^n m_i\ap_i \mid m_i=0,1,2,\dots \}$ is the monoid generated by the 
positive roots.

\noindent
As usual, $\mu^\vee=2\mu/(\mu,\mu)$ is the coroot
for $\mu\in \Delta$ and $Q^\vee=\oplus _{i=1}^n {\mathbb Z}\ap_i^\vee$  
is the coroot lattice. 
\\[.6ex]
Letting $\widehat V=V\oplus {\mathbb R}\delta\oplus {\mathbb R}\lb$, we extend
the inner product $(\ ,\ )$ on $\widehat V$ so that $(\delta,V)=(\lb,V)=
(\delta,\delta)= (\lb,\lb)=0$ and $(\delta,\lb)=1$. Set  $\ap_0=\delta-\theta$.
Then 
\begin{itemize}
\item[] \ 
$\widehat\Delta=\{\Delta+k\delta \mid k\in {\mathbb Z}\}$ is the set of affine
(real) roots; 
\item[] \ $\HD^+= \Delta^+ \cup \{ \Delta +k\delta \mid k\ge 1\}$ is
the set of positive affine roots; 
\item[] \ $\HP=\Pi\cup\{\ap_0\}$ is the corresponding set
of affine simple roots. 
\end{itemize}
For each $\ap_i\in \HP$, let $s_i$ denote the corresponding reflection in $GL(\HV)$.
That is, $s_i(x)=x- (x,\ap_i)\ap_i^\vee$ for any $x\in \HV$.
The affine Weyl group, $\HW$, is the subgroup of $GL(\HV)$
generated by the reflections $s_0,s_1,\dots,s_n$.
The inner product $(\ ,\ )$ on $\widehat V$ is $\widehat W$-invariant. 
For $w\in\HW$, we set $\eus N(w)=\{\nu\in\HD^+\mid w(\nu)\in -\HD^+\}$.

Following D.\,Peterson, we say that $w\in \HW$ is  {\it minuscule\/}, if 
$\eus N(w)$ is of the form $\{\delta-\gamma\mid \gamma\in I_w\}$ 
for some subset $I_w\subset \Delta$.
It is not hard to prove that {\sf (i)} $I_w\subset \Delta^+$, {\sf (ii)} $I$ is an abelian ideal, and
{\sf (iii)} the assignment 
$w\mapsto I_w$ yields a bijection between the minuscule elements of
$\HW$ and the abelian ideals, see \cite{ko98},  \cite[Prop.\,2.8]{cp1}. 
Conversely, if $I\in\Ab$, then $w_I$ stands for the corresponding minuscule
element of $\HW$. 

{\bf (I)} \ Our first step is to assign an element of $Q^\vee$ to an abelian ideal. 
(This is known and,
moreover, such an assignment  can be performed for any \adn ideal of $\be$ \cite{cp2}.)
In fact, we first associate an element of $Q^\vee$ to any $w\in\HW$.

\noindent
Recall that $\HW$ is a semi-direct product of $W$ and $Q^\vee$, and it can be regarded as
a group of affine-linear transformations of $V$ \cite{hump}.
For any $w\in \HW$, we have a unique decomposition
\[
   w=v{\cdot}t_r,
\]
where $v\in W$ and $t_r$ is the translation of $V$ corresponding to $r\in Q^\vee$, i.e., 
$t_r(x)=x+r$ for all $x\in V$.
Now, we  assign an element of $Q^\vee$ to any $w\in\HW$ as follows:
\[
      w\mapsto v(r) .
\]
An alternative way for doing so (which does not appeal to the semi-direct product
structure) is the following. 
Define the integers $k_i$, $i=1,\dots,n$,  by the rule
$w^{-1}(\ap_i)=\mu_i+k_i\delta$ ($\mu_i\in \Delta$).
Then there is a unique element $z\in Q^\vee$ such that $(z,\ap_i)=k_i$.
It is easily seen that these two approaches, via the linear action on $\HV$ or the affine-linear 
action on $V$, are equivalent, i.e., $(\ap_i,v(r))=k_i$. 
If $w=w_I$ is minuscule, then we also write $z_I$ for the resulting element of $Q^\vee$. 
It is shown in \cite[Sect\,2]{ko98} that  the mapping
$I \mapsto z_I\in V$ sets up a bijection between $\Ab(\g)$
and   $\eus Z_1=\{ z\in Q^\vee \mid (z,\gamma)\in \{-1,0,1,2\}  \quad \forall \gamma\in \Delta^+\}$. 

{\bf (II)} \ Having constructed $z_I\in Q^\vee$, we write
\[
    z_I=\sum_{i=1}^n  m_i \ap_i^\vee, \quad  m_i\in \mathbb Z .
\] 
Finally, we define the subset $S_I$ of $\Pi$ as follows:
$S_I=\{ \ap_i\in \Pi \mid m_i \text{ is odd }\}$. In other words, the
Boolean vector   $(m_1\dots m_n) \pmod 2$ is the characteristic vector of $S_I$.

\begin{thm}  \label{thm:nat-bij}
The map $(I\in \Ab(\g)) \mapsto (S_I\subset \Pi)$ sets up a one-to-one correspondence  
between $\Ab(\g)$ and $2^\Pi$.
\end{thm}
\begin{proof}
By a result of Peterson (see \cite[Lemma~2.2]{ko98}),
\[
   \eus D=\{v\in V\mid -1< (x,\gamma)\le 1\quad \forall \gamma\in \Delta^+\}
\]
is a fundamental domain for the $Q^\vee$-action on $V$ by translations.
Let $\zeta: V \to V/Q^\vee\simeq (\mathbb S^1)^n$ be the quotient map.
Consider $\eus Z_1/2 \subset \eus D$. 
Clearly, $\eus Z_1/2 \to \zeta(\eus Z_1/2)\subset (\mathbb S^1)^n$ is bijective, and the image consists of all elements of order $2$.
Equivalently,  all  the subsets $S_I$ ($I\in\Ab$) are different.
\end{proof}

Unfortunately, this bijection does not behave well with respect to the number of generators
and the number of connected components, see Example~\ref{ex:A3}.

\begin{rmk}   \label{rem:parab-bij}
Let $\mathfrak{Par}(\g)$ be the set of all standard parabolic subalgebras of $\g$.
As is well known, there is a one-to-one correspondence 
\beq    \label{bij-parabolics}
  \mathfrak{Par}(\g) \stackrel{1:1}{\longleftrightarrow} 2^{\Pi} 
\eeq
that assigns to $\p\in \mathfrak{Par}(\g)$ the set of simple roots of the standard Levi subalgebra of $\p$.

On the other hand, there is a natural map $\Psi: \Ab(\g)\to \mathfrak{Par}(\g)$ that takes
an abelian ideal $\ah \subset \be$ to its normaliser in $\g$, denoted $\n_\g(\ah)$.
This map was studied in \cite{pr}, and it was proved there that
$\Psi$ is one-to-one {\sl if and only if\/} $\g$ is of type $\GR{A}{n}$ or $\GR{C}{n}$.
In particular, combining $\Psi$ with \eqref{bij-parabolics},
we obtain the \un{third} natural bijection $\Ab(\GR{A}{n}) \to 2^{\Pi}$.
It is remarkable that all three are different!
\end{rmk}

\begin{ex}     \label{ex:A3}
For $\Delta$ of type $\GR{A}{3}$,
we compare  three bijections given by Theorem~\ref{thm:An}, 
Theorem~\ref{thm:nat-bij},  and Remark~\ref{rem:parab-bij}. 
The first two columns of Table~\ref{table:A3} contain the input: the vector $z_I\in Q^\vee$ 
and the set
of generators of $I$. In the first (resp. second) column, a triple
$m_1m_2m_3$ stands
for $m_1\ap^\vee_1+m_2\ap^\vee_2+m_3\ap^\vee_3$ (resp. 
$m_1\ap_1+m_2\ap_2+m_3\ap_3$).
The third column gives the characteristic vector of $S_I$.
The last column shows the simple roots of the standard Levi subalgebra of $\n_\g(\ah_I)$,
where $\ah_I$ is the ideal of $\be$ corresponding to $I$.
One sees that $z_I\pmod 2$ and $\Phi(I)$ differ in the last two rows, and the last column 
is different from the previous two (even if we take the complement!).

\begin{table}[htb]    
\caption{ Three bijections for $\GR{A}{3}$}  \label{table:A3} 
\begin{tabular}{cc|ccc}
$z_I$ & $\Gamma(I)$ & $z_I\pmod 2$ &  $\Phi(I)$  & Levi of $\n_\g(\ah_I)$ \\ \hline
000 &  $\varnothing$ & 000 &   $\varnothing$ & \{1,2,3\}\\
111  &  $\{ 111 \}$        &  111  &   \{1,2,3\}  &  \{2\}  \\
110  &  $\{ 110 \}$        &  110 &   \{1,2\}     &  \{3\}  \\
011  &  $\{ 011 \}$        &  011 &   \{2,3\}     & \{1\}   \\
100  &  $\{ 100 \}$       &  100  &   \{1\}       & \{2,3\} \\
001  &  $\{ 001 \}$        &  001 &   \{3\}       & \{1,2\} \\
010  &  $\{ 110, 011 \}$ &  010 &  \{1,3\}    & $\varnothing$ \\
121  &  $\{ 010 \}$        &  101 &   \{2\}       &  \{1,3\} \\  
\hline
\end{tabular} 
\end{table}
\end{ex}

\end{document}